\documentclass{article}
\usepackage{maa-monthly-alt}

\theoremstyle{theorem}
\newtheorem{theorem}{Theorem}
\newtheorem{thm}[theorem]{Theorem}
\newtheorem{lem}[theorem]{Lemma}
\newtheorem{prop}[theorem]{Proposition}
\newtheorem{cor}[theorem]{Corollary}

\theoremstyle{definition}

\begin{document}

\title{How many units can a commutative ring have?}
\markright{How many units can a commutative ring have?}
\author{Sunil K. Chebolu and Keir Lockridge}

\maketitle

\final

\begin{abstract}
L\'{a}szl\'{o} Fuchs posed the following problem in 1960, which remains open: classify the abelian groups occurring as the group of all units in a commutative ring. In this note, we provide an elementary solution to a simpler, related problem: find all cardinal numbers occurring as the cardinality of the group of all units in a commutative ring. As a by-product, we obtain a solution to Fuchs' problem for the class of finite abelian $p$-groups when $p$ is an odd prime. 
\end{abstract}

\section{Introduction}\label{sec:introduction}

It is well known that a positive integer $k$ is the order of the multiplicative group of a finite field if and only if $k$ is one less than a prime power.  The corresponding fact for finite commutative rings, however, is not as well known. Our motivation for studying this---and the question raised in the title---stems from a problem posed by L\'{a}szl\'{o} Fuchs in 1960: characterize the abelian groups that are the group of units in a commutative ring (see \cite{Fuchs}). Though Fuchs' problem remains open, it has been solved for various specialized classes of groups, where the ring is not assumed to be commutative. Examples include cyclic groups (\cite{PS}), alternating, symmetric and finite simple groups (\cite{CT-1, CT-2}), indecomposable abelian groups (\cite{cl-2}), and dihedral groups (\cite{cl-3}). In this note we consider a much weaker version of Fuchs' problem, determining only the possible cardinal numbers $|R^\times|$, where $R$ is a commutative ring with group of units $R^\times$.

In a previous note in the \textsc{Monthly} (\cite{ditor}), Ditor showed that a finite group $G$ of odd order is the group of units of a ring if and only if $G$ is isomorphic to a direct product of cyclic groups $G_i$, where $|G_i| = 2^{n_i} - 1$ for some positive integer $n_i$. This implies that an odd positive integer is the number of units in a ring if and only if it is of the form $\prod_{i=1}^t (2^{n_i} -1) $ for some positive integers $n_1, \dots, n_t$. As Ditor mentioned in his paper, this theorem may be derived from the work of  Eldridge in \cite{Eldridge} in conjunction with the Feit-Thompson theorem which says that every finite group  of odd order is solvable. However, the purpose of his note was to give an elementary proof of this result using  classical structure theory. Specifically, Ditor's proof uses the following key ingredients: Mashke's theorem, which classifies (for finite groups) the group algebras over a field that are semisimple rings; the Artin-Wedderburn theorem, which describes the structure of semisimple rings; and Wedderburn's little theorem, which states that every finite division ring is a field.

In this note we give another elementary proof of Ditor's theorem for commutative rings. We also extend the theorem to even numbers and infinite cardinals,  providing a complete answer to the  question posed in the title; see Theorem \ref{maintheorem}. Our approach also gives an elementary solution to Fuchs' problem for finite abelian $p$-groups when $p$ is an odd prime; see Corollary \ref{p-groups-odd}.

\section{Finite cardinals}

We begin with two lemmas. For a prime $p$, let $\mathbf{F}_p$ denote the field of $p$ elements. Recall that any ring homomorphism $\phi \colon A \longrightarrow B$ maps units to units, so $\phi$ induces a group homomorphism $\phi^\times \colon  A^\times \longrightarrow B^\times$. 

\begin{lem} \label{units-quotient}
Let $\phi \colon A \longrightarrow B$  be a homomorphism of commutative rings. If the induced group homomorphism $\phi^\times \colon  A^\times \longrightarrow B^\times$ is surjective, then there is a quotient $A'$ of $A$ such that $(A')^\times \cong B^\times$.
\end{lem} 

\begin{proof}
By the first isomorphism theorem for rings, $\mathrm{Im}\, \phi$ is isomorphic to a quotient of $A$. It therefore suffices to prove that $(\mathrm{Im}\, \phi)^\times = B^\times$. Since ring homomorphisms map units to units and $\phi^\times$ is surjective, we have $B^\times \subseteq (\mathrm{Im}\, \phi)^\times$. The reverse inclusion $(\mathrm{Im}\, \phi)^\times \subseteq B^\times$ holds since every unit in the subring $\mathrm{Im}\, \phi$ must also be a unit in the ambient ring $B$.
\end{proof}

\begin{lem}\label{ff-tensor}
Let $V$ and $W$ be finite fields of characteristic $2$.
\begin{enumerate}
\item The tensor product $V \otimes_{\mathbf{F}_2} W$ is isomorphic as a ring to a finite direct product of finite fields of characteristic $2$. \label{fdpff}
\item As $\mathbf{F}_2$-vector spaces, $\dim(V \otimes_{\mathbf{F}_2} W) = (\dim V)(\dim W)$. \label{dims}
\end{enumerate}
\end{lem}
\begin{proof}
To prove (\ref{fdpff}), let $K$ and $L$ be finite fields of characteristic 2.  By the primitive element theorem,  we have $L \cong \mathbf{F}_2[x]/(f(x))$, where $f(x)$ is an irreducible polynomial  in $\mathbf{F}_2[x]$. This implies that
\[K \otimes_{\mathbf{F}_2} L \cong \frac{K[x]}{(f(x))}.\]
The irreducible factors of $f(x)$ in $K[x]$ are distinct since the extension $L/\mathbf{F}_2$ is separable. Now let $f(x) = \prod_{i=1}^t f_i(x)$ be the factorization of $f(x)$ into its distinct irreducible factors in $K[x]$. We then have the following series of ring isomorphisms:
\[K \otimes_{\mathbf{F}_2} L \cong  \frac{K[x]}{(f(x))} \cong \frac{K[x]}{(\prod_{i=1}^t f_i(x))} \cong \prod_{i=1}^t \frac{K[x]}{(f_i(x))}, \]
where the last isomorphism follows from the Chinese remainder theorem for the ring $K[x]$.
Since each factor $K[x]/(f_i(x))$ is a finite field of characteristic $2$, we see that $K \otimes_{\mathbf{F}_2} L$ is isomorphic as a ring to a direct product of finite fields of characteristic $2$, as desired.

For (\ref{dims}), simply note that if $\{v_1, \dots, v_k\}$ is a basis for $V$ and $\{w_1, \dots, w_l\}$ is a basis for $W$, then $\{v_i \otimes w_j \, | \, 1 \leq i \leq k, 1 \leq j \leq l\}$ is a basis for $V \otimes_{\mathbf{F}_2}W$.
\end{proof}

We may now classify the finite abelian groups of odd order that appear as the group of units in a commutative ring.
 
\begin{prop} \label{finite-odd} Let $G$ be a finite abelian group of odd order. The group $G$ is isomorphic to the group of units in a commutative ring if and only if $G$ is isomorphic to the group of units in a finite direct product of finite fields of characteristic $2$. In particular, an odd positive integer $k$ is the number of units in a commutative ring if and only if $k$ is of the form $\prod_{i = 1}^t (2^{n_i} -1)$ for some positive integers $n_1, \dots, n_t$.

\end{prop}
\begin{proof}
The `if' direction of the second statement follows from the fact that, for rings $A$ and $B$, $(A \times B)^\times \cong A^\times \times B^\times$.

For the converse, since the trivial group is the group of units of $\mathbf{F}_2$, let $G$ be a nontrivial finite abelian group of odd order and let $R$ be a commutative ring with group of units $G$. Since $G$ has odd order, the unit $-1$ in $R$ must have order 1. This implies that $R$ has characteristic 2.

Now let $$G \cong \mathbf{C}_{p_1^{\alpha_1}} \times \cdots \times \mathbf{C}_{p_k^{\alpha_k}}$$ denote a decomposition of $G$ as a direct product of cyclic groups of prime power order (the primes $p_i$ are not necessarily distinct). Let $g_i$ denote a generator of the $i$th factor. Define a ring $S$ by
\[
S = \frac{\mathbf{F}_2[x_1, \dots, x_k]}{(x_1^{p_1^{\alpha_1}} - 1, \dots, x_k^{p_k^{\alpha_k}} - 1)}.
 \]

\noindent Since $R$ is a commutative ring of characteristic 2, there is a natural ring homomorphism $S \longrightarrow R$ sending $x_i$ to $g_i$ for all $i$.  Since the $g_i$'s together generate $G$, this map induces a surjection $S^\times \longrightarrow R^\times,$ and hence by Lemma \ref{units-quotient} there is a quotient of $S$ whose group of units is isomorphic to $G$.

Since any quotient of a finite direct product of fields is again a finite direct product  of fields (of possibly fewer factors), the proof will be complete if we can show that $S$ is isomorphic as a ring to a finite direct product of fields of characteristic $2$. To see this, observe that the map
\[
\mathbf{F}_2[x_1]/(x_1^{p_1^{\alpha_1}} - 1) \times \cdots \times \mathbf{F}_2[x_1]/(x_k^{p_k^{\alpha_k}} - 1) \longrightarrow S
\]
\noindent sending a $k$-tuple to the product of its entries is surjective and $\mathbf{F}_2$-linear in each factor; by the universal property of the tensor product, it induces a surjective ring homomorphism
\begin{equation}\label{tensor-iso}\tag{\dag}
\mathbf{F}_2[x_1]/(x_1^{p_1^{\alpha_1}} - 1) \otimes_{\mathbf{F}_2} \cdots \otimes_{\mathbf{F}_2} \mathbf{F}_2[x_1]/(x_k^{p_k^{\alpha_k}} - 1) \longrightarrow S.
\end{equation}
\noindent The dimension of the source of (\ref{tensor-iso}) as an $\mathbf{F}_2$-vector space is $p_1^{\alpha_1}\cdots p_k^{\alpha_k}$ by Lemma \ref{ff-tensor} (\ref{dims}); this is also the dimension of the target (count monomials in the polynomial ring $S$). Consequently, the map (\ref{tensor-iso}) is an isomorphism of rings. The irreducible factors of each polynomial $x_i^{p_i^{\alpha_i}} - 1$ are distinct since this polynomial has no roots in common with its derivative ($p_i$ is odd).  Therefore by the Chinese remainder theorem, each tensor factor is a finite direct product of finite fields of characteristic 2. Since the tensor product distributes over finite direct products,  we may use  Lemma \ref{ff-tensor} (\ref{fdpff}) to conclude that $S$ is ring isomorphic to a finite direct product of finite fields of characteristic 2.
\end{proof}

For any odd prime $p$, we now characterize the finite abelian $p$-groups that are realizable as the group of units of a commutative ring. Recall that a finite $p$-group is a finite group whose order is a power of $p$. An elementary abelian finite $p$-group is a finite group that is isomorphic to a finite direct product of cyclic groups of order $p$.

\begin{cor}\label{p-groups-odd}
Let $p$ be an odd prime.  A finite abelian $p$-group $G$ is the group of units of a commutative ring if and only if $G$ is an elementary abelian $p$-group and $p$ is a Mersenne prime.
\end{cor}
\begin{proof}
The `if' direction follows from the fact if $p = 2^n - 1$ is a Mersenne prime, then $$(\mathbf{F}_{p+1} \times \cdots \times \mathbf{F}_{p+1})^\times \cong \mathbf{C}_p \times \cdots \times \mathbf{C}_p.$$ For the other direction, let $p$ be an odd prime and let $G$ be a finite abelian $p$-group. If $G$ is the group of units of commutative ring, then  by Proposition \ref{finite-odd}, $G \cong T^\times$ where $T$ is a finite direct product of finite fields of characteristic 2. Consequently, $$G \cong \mathbf{C}_{2^{n_1} - 1} \times \cdots \times \mathbf{C}_{2^{n_t} - 1}.$$ Since each factor must be a $p$-group, for each $i$ we have $2^{n_i} - 1 = p^{z_i}$ for some positive integer $z_i$. 
We claim that $z_i = 1$ for all $i$. This follows from  \cite[2.3]{cl-1}, but since the argument is short we include it here for convenience. 

Assume to the contrary that $z_i > 1$ for some $i$.  Consider the equation $p^{z_i} + 1 = 2^{n_i}$. Since $p > 1$, we have $n_i \ge 2$ and hence $p^{z_i} \equiv -1  \mod 4$. This means $p \equiv -1 \mod 4$ and $z_i$ is odd. Since $z_i > 1$, we have a nontrivial factorization 
\[ 2^{n_i} = p^{z_i} + 1 = (p+1)(p^{z_i-1} - p^{z_i-2} + \cdots -p +1),\]
so $p^{z_i-1} - p^{z_i-2} + \cdots -p +1$ must be even. On the other hand, since $z_i$ and $p$ are both odd, working modulo 2 we obtain 
\[
0 \equiv p^{z_i-1} - p^{z_i-2} + \cdots -p +1 \equiv z_i \equiv 1 \mod 2,
\]
\noindent a contradiction. Hence $z_i = 1$ for all $i$, so $p$ is Mersenne and $G$ an is elementary abelian $p$-group.
\end{proof}

\noindent The above corollary does not hold for the Mersenne prime $p =2$; for example, $\mathbf{C}_4 = \mathbf{F}_5^\times$. As far as we know, Fuchs' problem for finite abelian $2$-groups remains open.


We next provide a simple example demonstrating that every even number is the number of units in a commutative ring.

\begin{prop}\label{finite-even} Every even number is the number of units in a commutative ring. 
\end{prop}
\begin{proof}
Let $m$ be a positive integer and consider the commutative ring
\[ R_{2m} =  \frac{\mathbf{Z}[x]}{(x^2, mx)}.\]
Every element in this ring can be uniquely represented by an element of the form $a+bx$, where $a$ is an arbitrary integer and  $0 \le b \le m-1$. We will now show that $a+bx$ is a unit in this ring if and only if $a$ is either $1$ or $-1$; this implies the ring has exactly $2m$ units.  (In fact, it can be shown that $R_{2m}^\times \cong \mathbf{C}_2 \times \mathbf{C}_m$.)

If $a+bx$ is a unit in $R_{2m}$, there there exits an element $a'+b'x$ such that $(a+bx)(a'+b'x) = 1$ in $R_{2m}$.  Since $x^2=0$ in $R_{2m}$, we must have $aa'  = 1 $ in $\mathbf{Z}$; i.e., $a$ is  $1$ or $-1$.  Conversely, if $a$ is  $1$ or $-1$, we see that $(a + bx)(a - bx) = 1$ in $R_{2m}$.
\end{proof}

\section{Infinite cardinals}

Propositions \ref{finite-odd} and \ref{finite-even} solve our problem for finite cardinals. For infinite cardinals, we will prove the following proposition.

\begin{prop} \label{infinite}
Every infinite cardinal is the number of units in a commutative ring.
\end{prop}

\noindent Our proof relies mainly on the Cantor-Bernstein theorem:
\begin{thm}[Cantor-Bernstein]
Let $A$ and $B$ be any two  sets. 
 If there exist injective mappings $f \colon A \longrightarrow B$ and $g \colon B \longrightarrow A$, then there exits a bijective mapping $h \colon A \longrightarrow B$. In other words, if $|A| \le |B|$ and $|B| \le |A|$, then $|A| = |B|$. 
\end{thm}
\noindent We also use other standard facts from set theory which may be found in \cite{Halmos}. For example, we make frequent use of the fact that whenever $\alpha$ and $\beta$ are infinite cardinals with $\alpha \leq \beta$, then $\alpha\beta \leq \beta$. Recall that $\aleph_0$ denotes the cardinality of the set of natural numbers.

\begin{proof}[Proof of Proposition \ref{infinite}] Let $\lambda$ be an  infinite cardinal and let $S$ be a set whose cardinality is $\lambda$. Consider $\mathbf{F}_2(S)$, the field of rational functions in the elements of $S$. We claim that $|\mathbf{F}_2(S)^\times| = \lambda$. By the Cantor-Bernstein theorem, it suffices to prove that  $|S| \le |\mathbf{F}_2(S)^\times|$ and $|\mathbf{F}_2(S)^\times| \le |S| $. Since $S \subseteq \mathbf{F}_2(S)^\times$, it is clear that $ |S| \le |\mathbf{F}_2(S)^\times|$.

For the reverse inequality, first observe that if $A$ is a finite set, then $|\mathbf{F}_2[A]| = \aleph_0$. This follows by induction on the size of $A$, because $\mathbf{F}_2[x]$ is countable and $R[x]$ is countable whenever $R$ is countable. We now have the following:

\begin{eqnarray*}
|\mathbf{F}_2(S)^\times| & \le & |\mathbf{F}_2(S)| \\
                                 & \le & |\mathbf{F}_2[S] \times \mathbf{F}_2[S]|  \ \ \text{(every rational function is a ratio of  two polynomials)}\\
                                 & = & |\mathbf{F}_2[S]|^2 \\
                                 & = & |\mathbf{F}_2[S]|  \\\
                                                                  & \le & \sum_{A \subset S,\; 1 \le |A| < \aleph_0} |\mathbf{F}_2[A]| \\
                                 & = &  \sum_{A \subset S, \; 1\le  |A| < \aleph_0} \aleph_0 \\                                 & \le & \sum_{i=1}^{\infty} |S|^i \aleph_0 =  \sum_{i=1}^{\infty} |S| \aleph_0   =  |S| \aleph_0^2 =  |S| \aleph_0 = |S|.
\end{eqnarray*} 

\end{proof}

Combining Propositions \ref{finite-odd}, \ref{finite-even}, and \ref{infinite}, we obtain our main result.

\begin{thm} \label{maintheorem}
Let $\lambda$ be a cardinal number. There exists a commutative ring $R$ with $|R^\times| = \lambda$ if and only if $\lambda$ is equal to
\begin{enumerate}
\item an odd number of the form $\prod_{i=1}^t (2^{n_i} -1) $ for some positive integers $n_1, \dots, n_t$,
\item an even number, or
\item an infinite cardinal number.
\end{enumerate} \end{thm}

\begin{acknowledgment}{Acknowledgments.}
We would like to thank George Seelinger for simplifying our presentation of a ring with an even number of units. We also would like to thank the anonymous referees for their comments. 
\end{acknowledgment}

\begin{affil}
Department of Mathematics, Illinois State University, Normal, IL 61790, USA\\
schebol@ilstu.edu
\end{affil}
\begin{affil}
Department of Mathematics, Gettysburg College, Gettysburg, PA 17325, USA\\
klockrid@gettysburg.edu
\end{affil}
\vfill\eject

\end{document}